\documentclass[draft]{birkjour}
\usepackage[noadjust]{cite}
\usepackage{xcolor}
\usepackage{mathrsfs}

\usepackage{amsmath}
\usepackage{amssymb}
\usepackage{amsthm}
\usepackage{extarrows}
\RequirePackage[all]{xy}

\newtheorem{theorem}{Theorem}[section]

\newtheorem{proposition}[theorem]{Proposition}
\theoremstyle{definition}
\newtheorem{definition}[theorem]{Definition}
\theoremstyle{remark}
\newtheorem*{remark}{Remark}

\numberwithin{equation}{section}

\newcommand{\BibTeX}{B\kern-0.1emi\kern-0.017emb\kern-0.15em\TeX}
\newcommand{\XYpic}{$\mathrm{X\kern-0.3em\raisebox{-0.18em}{Y}}$-$\mathrm{pic}\,$}

\newcommand{\cl}{C \kern -0.1em \ell}  

\newcommand{\So}{\mathbb{S}}

\newcommand{\bx}{\boldsymbol{x}}
\newcommand{\boy}{\boldsymbol{y}}

\newcommand{\ubx}{\underline{\bx}}
\newcommand{\ueta}{\underline{\eta}}
\newcommand{\uomega}{\underline{\omega}}
\newcommand{\ugamma}{\underline{\gamma}}

\newcommand{\be}{\begin{eqnarray*}}
	\newcommand{\ee}{\end{eqnarray*}}

\allowdisplaybreaks

\newcommand{\ed}{\end{document}}
\begin{document}

%
%
%
%
%
%
%
%
%

\title[Teodorescu transform for generalized partial slice monogenic functions]
{Integral formulas and Teodorescu transform for generalized partial-slice monogenic functions}
\author[Manjie Hu]{Manjie Hu}
\address{%
	School of Mathematical Sciences,\\ Anhui University, Hefei, Anhui, China}
\email{A23201033@stu.ahu.edu.cn}

\author[Chao Ding]{Chao Ding}
%
\address{%
	Center for Pure Mathematics, \\School of Mathematical Sciences,\\ Anhui University, Hefei, Anhui, China}
\email{cding@ahu.edu.cn}

\author[Yifei Shen]{Yifei Shen}
%
\address{%
	Stony Brook Institute at Anhui University, Hefei, Anhui, China}
\email{R22314014@stu.ahu.edu.cn}  

\author[Jiani Wang]{Jiani Wang}
%
\address{%
	Stony Brook Institute at Anhui University, Hefei, Anhui, China}
\email{R22314085@stu.ahu.edu.cn}
%
\date{\today}
\begin{abstract}
The theory of generalized partial-slice monogenic functions is considered as a syhthesis of the classical Clifford analysis and the theory of  slice monogenic functions. In this paper, we investigate the Cauchy integral formula and the Plemelj formula for generalized partial-slice monogenic functions. Further, we study some properties of the Teodorescu transform in this context. A norm estimation for the Teodorescu transform is discussed as well.
\end{abstract}
\label{page:firstblob}
\maketitle
\section{Introduction}\hspace*{\fill} 
Classical Clifford analysis is a mathematical field that extends the concepts of complex analysis to higher-dimensional spaces through the use of Clifford algebras. It centers around the study of monogenic functions (null solutions to the Dirac operator), which has been fully developed in the last decades. More details can be found, for instance, in \cite{Br,Co3,De,Gu,Gi}.
\par
A significant difference between classical Clifford analysis and complex analysis is that polynomials formulated with vectors in higher dimensions are no longer monogenic anymore due to the non-commutativity of multiplication of Clifford numbers. A consequence of this is that theory of series expansion is not as simple as in complex analysis. In 2010, Colombo, Sabadini, and Struppa \cite{Co1} extended this concept to the broader context of Clifford algebras, introducing the notion of slice monogenic functions. These functions were initially defined as those that are holomorphic on each slice within Euclidean space, serving as a localized definition for slice regular functions. More importantly, polynomials given in terms of vectors are monogenic in this context, and consequently, numerous findings pertaining to slice monogenic functions are formulated in terms of slices, including nicely formulated Taylor and Laurent series expansions. For instance, in \cite{Do}, the authors establish an identity principle for slice regular functions in slice topology domains, generalize slice function definitions to arbitrary $\mathbb{H}$ subsets. In \cite{Bi}, the author constructed three second-order differential operators with slice regular functions in their kernels, confirming that slice regular functions over $\mathbb{H}$ are harmonic in a sense. In \cite{Co6}, the authors present the construction of a regular and non-commutative Cauchy kernel for slice regular quaternionic functions, along with its formula yielding a novel Cauchy formula. \cite{Co} introduces a new monogenicity concept for Clifford algebra-valued functions from $\mathbb{R} ^{n+1}$ to $\mathbb{R} _n$, inspired by \cite{Ge} and \cite{Ge1}, and derives a corresponding Cauchy integral formula. 
\par
In 2013, Colombo et al. \cite{Co5} presented a global differential operator $G$ with non-constant coefficients, whose zero solutions exhibit a strong connection to slice monogenic functions under suitable domain conditions. Consequently, a theoretical framework centered around the differential operator $G$ has emerged, analogous to the important role of the Cauchy-Riemann operator in single-variable complex analysis. Subsequently, this global differential operator for slice regularity has gained significant research attention. For example, \cite{Gh1} establish a framework of global differential equations for slice regular functions over real alternative algebras, refining quaternionic and slice monogenic results and generalizing them.  They introduce differential operators $\vartheta$ and $\bar{\vartheta}$ extending $\partial /\partial x$ and $\partial /\partial x^c$ to $\Omega _D\backslash \mathbb{R}$, elucidating their expressions in quaternions, octonions, and Clifford algebras, and the connection to slice regularity.
\par
The theory of slice monogenic functions also has deep motivations in quantum mechanics. Back to 1930s, Birkhoff and von Neumann showed that quantum mechanics can be formulated over the real, the complex, and the quaternionic numbers. However, the spectral theory of quaternion linear operators has not been effectively developed for a long time. This is mainly because the understanding of the spectral concept of quaternion linear operators is not deep enough, resulting in the inability to effectively construct quaternion spectral theory. In 2006 , using techniques based solely on slice hyperholomorphic functions, the precise notion of spectrum of a quaternionic linear operator, also known as S-spectrum, was identified by F. Colombo and I. Sabadini. The proposal of this concept laid a solid foundation for the development of quaternion spectrum theory. Since then the literature in quaternionic spectral theory has rapidly grown, and much work has contributed to this topic, see e.g. \cite{A,B,C,D,E,F,G}.
\par
In \cite{Xu,Xu1}, the authors introduce a novel class of functions and operators, notably incorporating classical monogenic functions and slice monogenic functions as two particular instances. The main goal of this article is to develop integral formulas and Teodorescu transform in theory of generalized partial-slice monogenic functions. This provides important tools to study some differential equations (such as Beltrami equations) and Dirichlet problems related to partial-slice monogenic functions.
\par
This paper is organized as follows. In Section 2, we comprehensively review the fundamental principles of Clifford algebras and delve into the basic definitions of monogenic and slice monogenic functions. Section 3 is devoted to an introduction to the operator notations and pivotal theorems pertaining to Generalized slice monogenic functions, which are central to our investigation. Extensions for the newly defined Teodorescu operator to various important theorems and corollaries are given in Section 4. 
\section{Preliminaries}
In this section, we review some definitions and notations in Clifford algebras, as well as the theories of monogenic functions and slice monogenic functions. More details can be found in \cite{Br,Co2,De,Gent}.
\subsection{Clifford algebras}
Let $\left\{ e_1,e_2,\cdots ,e_n \right\}$ be the standard orthonormal basis for the $n$-dimensional real Euclidean space $\mathbb{R} ^n$. The real Clifford algebra, denoted by $\mathbb{R} _n$, is constructed from these basis elements with the relations
$
	e_ie_j+e_je_i=-2\delta _{ij},\ 1\leqslant i,j\leqslant n,
$
where $\delta _{ij}$ is the Kronecker symbol. Consequently, any element $a\in\mathbb{R} _n$ can be expressed in the form
$
	a=\sum_A{a_Ae_A},\  a_A\in \mathbb{R} ,
$
where $e_A=e_{j_1}e_{j_2}\cdots e_{j_r},$ $A=\left\{ j_1,j_2,\cdots ,j_r \right\} \subseteq \left\{ 1,2,\cdots ,n \right\}$ and $1\leqslant j_1<j_2<\cdots <j_r\leqslant n$, $e_{\emptyset}=e_0=1$.
\par
For each integer $k=0,1,\cdots ,n$, the real linear subspace of $\mathbb{R} _n$, denoted by $\mathbb{R} _{n}^{k}$ and referred to as $k$-vectors, is generated by the binomial coefficient $\big( \begin{array}{c}n\\k\\\end{array}\big)$ elements of the form 
$
	e_A=e_{i_1}e_{i_2}\cdots e_{i_k},1\leqslant i_1<i_2<\cdots <i_k\leqslant n.
$
In particular, the element $\left[ a \right] _0=a_{\emptyset}$ is called the scalar part of $a$. Further, the so-called paravectors, which is a crucial subset of Clifford numbers $\mathbb{R} _n$, are elements in $\mathbb{R} _{n}^{0}\oplus \mathbb{R} _{n}^{1}$. This subset is subsequently associated with $\mathbb{R} ^{n+1}$ through  the following mapping
\begin{align*}
	\mathbb{R} ^{n+1}&\longrightarrow\mathbb{R} _{n}^{0}\oplus \mathbb{R} _{n}^{1} ,\\
	\left( x_0,x_1,\cdots ,x_n \right) &\longmapsto x=x_0+\underline{x}=\sum_{i=0}^n{e_ix_i}.
\end{align*}
Now, we introduce some involutions in $\mathbb{R} _n$ as follows. 
\begin{itemize}
\item \textbf{Clifford conjugation:} The Clifford conjugation of $a=\sum_A{a_A{e_A}}\in\mathbb{R}_n$ is defined as
\begin{align*}
	\overline{a}=\sum_A{a_A\overline{e_A}},
\end{align*}
where $\overline{e_{j_1}\cdots e_{j_r}}=\overline{e_{j_1}}\cdots \overline{e_{j_r}}$, $\overline{e_j}=-e_j$, $1\leqslant j\leqslant n$, $\overline{e_0}=e_0=1$.
\item \textbf{Clifford reversion:} The Clifford reversion of $a=\sum_A{a_A{e_A}}\in\mathbb{R}_n$ is defined as 
\begin{align*}
	\tilde{a}=\sum_A{a_A\widetilde{e_A}},
\end{align*}
where $\widetilde{e_{j_1}\cdots e_{j_r}}=e_{j_r}\cdots e_{j_1}$ and $\widetilde{ab}=\tilde{b}\tilde{a}$ for $a,b\in \mathbb{R} _n$.
\end{itemize}~\par
The norm of $a\in\mathbb{R}_n$ is formally defined as $\left| a \right|=[a\overline{a}]_0=\left( \sum_A{\left| a_A \right|^2} \right) ^{\frac{1}{2}}$. In particular, the norm of a nonzero paravector $x$ is given by $\left| x \right|=\left( x\bar{x} \right) ^{\frac{1}{2}}$, and so its inverse is given by $x^{-1}=\frac{\bar{x}}{\left| x \right|^{2}}$. 
\subsection{Monogenic and slice monogenic functions}
Now, we review some definitions and preliminaries in the theory of monogenic functions as well as those in the theory of slice monogenic functions. One can find more details in \cite{Br,Co2,De,Gent}.
\begin{definition}[Monogenic functions]
	Let $\Omega\subset\mathbb{R}^{n+1}$ be a domain and $f\,:\, \Omega \longrightarrow \mathbb{R}$ be a real-valued differentiable function. A function $f = \sum_A{e_A f_A}$ is \emph{left monogenic} in $\Omega$ if it satisfies the generalized Cauchy-Riemann equation
	$
		Df\left( x \right) =\sum_{i=0}^n{e_i\frac{\partial}{\partial x_i}}f(\bx)=0,\ x\in \Omega,
$
where $D$ is called the generalized Cauchy-Riemann operator. Since multiplications of Clifford numbers are not commutative in general, there is a similar definition for right monogenic.
\end{definition}
 \begin{remark}
Comparing with the generalized Cauchy-Riemann operator, the operator $\partial _{\underline{x}}=\sum_{i=1}^ne_i\frac{\partial}{\partial x_i}$  is known as the classical Dirac operator in $\mathbb{R} ^n$. In Clifford analysis, both the null-solutions of the generalized Cauchy-Riemann operator and those of the Dirac operator are known as monogenic functions.
\end{remark}
One can observe that every non-real paravector in the space $\mathbb{R} ^{n+1}$ can be expressed in the form of $$x=x_0+\sum_{i=1}^nx_ie_i=:x_0+\underline{x}=x_0+r\omega,$$ where $r=\left| \underline{x} \right|=\left( \sum_{i=1}^n{x_{i}^{2}} \right) ^{\frac{1}{2}}$ and $\omega=\dfrac{\underline{x}}{|\underline{x}|}$, which is uniquely specified with properties similar to a classical imaginary unit. In other words,
\begin{align*}
	\omega \in \mathbb{S}^{n-1}=\left\{ x\in \mathbb{R} ^{n+1}: x^2=-1 \right\} .
\end{align*}
When $x$ is real, then $r=0$ and for every $\omega \in \mathbb{S}^{n-1}$ one can write $x=x+\omega \cdot 0$.
\begin{definition}[Slice monogenic functions] Let $\Omega\subset\mathbb{R}^{n+1}$ be a domain and a function $f\,:\, \Omega \longrightarrow \mathbb{R} ^n$ is  \emph{(left) slice monogenic} if, for every direction $\omega \in S^{n-1}$, the function $f_{\omega}$ restricted to the subset $\Omega_{\omega} = \Omega \cap (\mathbb{R} \oplus \omega \mathbb{R}) \subseteq \mathbb{R}^2$ exhibits holomorphicity, meaning it possesses continuous partial derivatives and fulfills the condition
	\begin{align*}
		\left( \partial _{x_0}+\omega \partial _r \right) f_{\omega}\left( x_0+r\omega \right) =0
	\end{align*}
	for all points $x_0+r\omega \in \Omega _{\omega}$.
\end{definition}
There is also an alternative approach to define slice monogenic functions given by Ghiloni and Perotti \cite{Gh0} in 2011 with the concept of stem functions. Many works in the theory of slice monogenic functions have been done with this approach, for instance, \cite{Di,Gh1,Pe}.  We will review this in a more general setting in the following section.
\section{Generalized partial-slice monogenic functions}
In \cite{Xu,Xu1}, the authors generalized slice monogenic functions to the so-called generalized partial-slice monogenic functions. In particular, they pointed out that the theory of generalized partial-slice monogenic functions is a synthesis of classical Clifford analysis and the theory of slice monogenic functions. Here, we review some definitions and properties needed for the rest of this article. For more systematic details, we refer the reader to \cite{Xu,Xu1}.
\par
Let $p$ and $q$ be non-negative and positive integers, respectively. We consider functions $f\,:\,\Omega \longrightarrow \mathbb{R} _{p+q}$, where $\Omega\subset\mathbb{R}^{p+q+1}$ is a domain. An element $\boldsymbol{x} \in \mathbb{R}^{p+q+1}=\mathbb{R}^{p+1} \oplus \mathbb{R}^q$, can be identified with a paravector in $\mathbb{R}_{p+q}$ in the following manner
\begin{align*}
	\boldsymbol{x}=\boldsymbol{x}_p+\underline{\boldsymbol{x}}_q\in \mathbb{R} ^{p+1}\oplus \mathbb{R}^q,\ \boldsymbol{x}_p=\sum_{i=0}^p{x_ie_i},\ \underline{\boldsymbol{x}}_q=\sum_{i=p+1}^{p+q}{x_ie_i}.
\end{align*}
In this context, we define the generalized Cauchy-Riemann operator and the Euler operator as 
\begin{align*}
	&D_{\boldsymbol{x}}=\sum_{i=0}^{p+q}{e_i\partial_{x_i}}=\sum_{i=0}^p{e_i\partial_{x_i}}+\sum_{i=p+1}^{p+q}{e_i\partial_{x_i}} =:D_{\boldsymbol{x}_p}+D_{\underline{\boldsymbol{x}}_q},\\
	&\mathbb{E}_{\boldsymbol{x}}=\sum_{i=0}^{p+q}{x_i\partial_{x_i}}=\sum_{i=0}^p{x_i\partial _{x_i}}+\sum_{i=p+1}^{p+q}{x_i\partial_{x_i}}=: \mathbb{E} _{\boldsymbol{x}_p}+\mathbb{E}_{\underline{\boldsymbol{x}}_q}.
\end{align*}
Note that in these definitions, we have utilized the notation $D_{\boldsymbol{x}_p}$ and $D_{\underline{\boldsymbol{x}}_q}$ to represent the operators acting on the components $\boldsymbol{x}_p$ and $\underline{\boldsymbol{x}}_q$ respectively, and similarly for the Euler operators.
\par
We use $\mathbb{S}$ to represent the unit sphere in $\mathbb{R}^q$, where its elements is denoted by $\underline{\boldsymbol{x}}_q = \sum_{i=p+1}^{p+q} x_i e_i$, fulfilling the condition
\begin{align*}
	\mathbb{S} =\left\{ \underline{\boldsymbol{x}}_q:{\underline{\boldsymbol{x}}_q}^2=-1 \right\} =\left\{ \underline{\boldsymbol{x}}_q=\sum_{i=p+1}^{p+q}{x_ie_i}:\sum_{i=p+1}^{p+q}{{x_i}^2}=1 \right\}.
\end{align*}
It is noteworthy that for any non-zero $\underline{\boldsymbol{x}}_q$, there exists a unique positive real number $r \in \mathbb{R}^+$ and a unique unit vector $\underline{\omega} \in \mathbb{S}$, which represents the direction of $\underline{\boldsymbol{x}}_q$, such that $\underline{\boldsymbol{x}}_q = r\underline{\omega}$, where
\begin{align*}
	r=\left| \underline{\boldsymbol{x}}_q \right|,\underline{\omega}=\frac{\underline{\boldsymbol{x}}_q}{\left| \underline{\boldsymbol{x}}_q \right|}.
\end{align*}
If $\underline{\boldsymbol{x}}_q = 0$, we set $r = 0$, and $\underline{\omega}$ is not uniquely determined, in other words, $\underline{\boldsymbol{x}}_q = \boldsymbol{0} = \boldsymbol{x}_p + \underline{\omega} \cdot 0$ for arbitrary $\underline{\omega} \in \mathbb{S}$. 
\par
For an open set $\Omega \subset \mathbb{R}^{p+q+1}$, we introduce the notation
\begin{align*}
	\Omega _{\underline{\omega}}:=\Omega \cap \left( \mathbb{R} ^{p+1}\oplus \underline{\omega}\mathbb{R} \right) \subseteq \mathbb{R} ^{p+2},
\end{align*}
which defines a subset of $\mathbb{R}^{p+2}$ by intersecting $\Omega$ with the subspace spanned by $\mathbb{R}^{p+1}$ and the line through the origin determined by $\underline{\omega}$. In order to develop a theory of generalized partial-slice monogenic functions, we also need some restrictions on the domains considered.
\begin{definition}
		Let $\Omega$ be a domain in $\mathbb{R} ^{p+q+1}$, and $\Omega$ is called \emph{partially symmetric} with respect to $\mathbb{R}^{p+1}$ (or $p$-symmetric for short) if, for $\boldsymbol{x}_p \in \mathbb{R}^{p+1}$, $r \in \mathbb{R}^+$, and $\underline{\omega} \in \mathbb{S}$,$$\boldsymbol{x}=\boldsymbol{x}_p+r\underline{\omega}\in \Omega \Longrightarrow \left[ \boldsymbol{x} \right] :=\boldsymbol{x}_p+r\mathbb{S} =\left\{ \boldsymbol{x}_p+r\underline{\omega}, \underline{\omega}\in \mathbb{S} \right\} \subseteq \Omega .$$
\end{definition}
Now, we review the approach to study the theory of generalized partial-slice functions with the concept of stem functions as follows.
\begin{definition}
	A function $F$ : $D\longrightarrow \mathbb{R} _{p+q}\otimes _{\mathbb{R}}\mathbb{C}$ in an open set $D\subseteq \mathbb{R} ^{p+2}$, which is invariant under the reflection of the $(p+2)\text{-}th$ variable, is called a \emph{stem function} provided its $\mathbb{R}_{p+q}$-valued components, $F_1$ and $F_2$, where $F = F_1 + iF_2$, fulfill the conditions
	\begin{align}\label{Stem Function Costituent}
		F_1\left( \boldsymbol{x}_p,-r \right) =F_1\left( \boldsymbol{x}_p,r \right) ,\quad F_2\left( \boldsymbol{x}_p,-r \right) =-F_2\left( \boldsymbol{x}_p,r \right) ,\quad \left( \boldsymbol{x}_p,r \right) \in D.
	\end{align} 
	Each stem function $F$ induces a (left) generalized partial-slice function $f=\mathcal{I}(F)$ from $\Omega_D$ to $\mathbb{R}_{p+q}$, defined as $$f\left( \boldsymbol{x} \right) :=F_1\left( \boldsymbol{x}^{\prime} \right) +\underline{\omega}F_2\left( \boldsymbol{x}^{\prime} \right) ,\quad \boldsymbol{x}^{\prime}=\left( \boldsymbol{x}_p,r \right) ,\boldsymbol{x}=\boldsymbol{x}_p+r\underline{\omega}\in \mathbb{R} ^{p+q+1},\underline{\omega}\in \mathbb{S} .$$
\end{definition}

\begin{definition}
Let $D\subseteq \mathbb{R} ^{p+2}$ be a domain, which is invariant under the reflection of the $(p+2)\text{-}th$ variable. The \emph{p-symmetric completion} $\Omega_D\subset\mathbb{R}^{p+q}$ of $D$ is defined by
\begin{align*}
\Omega_D=\bigcup_{\underline{\omega}\in\mathbb{S}}\left\{x_p+r\underline{\omega}: \exists x_p\in\mathbb{R}^{p+1}, r\geq 0, (x_p,r)\in D\right\}.
\end{align*}
\end{definition}
Note that a domain $\Omega\subset\mathbb{R}^{p+q+1}$ is $p$-symmetric if and only if there exists a domain $D\subset\mathbb{R}^{p+2}$ such that $\Omega=\Omega_D$. In the rest of the article, we use $\Omega_D$ to stand for a $p$-symmetric domain in $\mathbb{R}^{p+q+1}$.
\par
Now, we denote the set of all induced generalized partial-slice functions on $\Omega _D$ by  $$\mathcal{G} \mathcal{S} \left( \Omega _D \right) :=\left\{ f=\mathcal{I} \left( F \right) : F\ \text{is a stem function on}\ D \right\} .$$
\begin{definition}
	Let $f \in \mathcal{G} \mathcal{S} (\Omega _D)$. The function $f$ is called \emph{generalized partial-slice monogenic of type $(p,q)$} if its stem function $F=F_1+iF_2$ satisfies the following generalized Cauchy-Riemann equations
	$$\begin{cases}  
		D_{\boldsymbol{x}_p}F_1-\partial _rF_2=0, \\  
		\overline{D_{\boldsymbol{x}_p}}F_2-\partial _rF_1=0.  
	\end{cases}  $$
\end{definition}
Similar as the case of slice functions, there is also a representation formula for generalized partial-slice functions of type $(p,q)$ as follows.
\begin{theorem}[Representation Formula]\label{*Representation Formula*}
	\cite{Xu} Let $f\in \mathcal{G} \mathcal{S} \left( \Omega _D \right)$. Then it holds that, for every $\boldsymbol{x}=\boldsymbol{x}_p+r\underline{\omega}\in \Omega _D$ with $\underline{\omega}\in \mathbb{S}$,
	\begin{align}
		f\left( \boldsymbol{x} \right)
		=&\left( \underline{\omega}-\underline{\omega}_2 \right) \left( \underline{\omega}_1-\underline{\omega}_2 \right) ^{-1}f\left( \boldsymbol{x}_p+r\underline{\omega}_1 \right)\nonumber\\
& -\left( \underline{\omega}-\underline{\omega}_1 \right) \left( \underline{\omega}_1-\underline{\omega}_2 \right) ^{-1}f\left( \boldsymbol{x}_p+r\underline{\omega}_2 \right) ,
	\end{align}
	for all $\underline{\omega}_1\ne \underline{\omega}_2\in \mathbb{S} $. In particular, if $\underline{\omega}_1=-\underline{\omega}_2 =\underline{\eta}\in \mathbb{S} $, we have
	\begin{align*}
		f\left( \boldsymbol{x} \right) &=\frac{1}{2}\left( 1-\underline{\omega}\underline{\eta} \right) f\left( \boldsymbol{x}_p+r\underline{\eta} \right) +\frac{1}{2}\left( 1+\underline{\omega}\underline{\eta} \right) f\left( \boldsymbol{x}_p-r\underline{\eta} \right)\\
		&=\frac{1}{2}\left( f\left( \boldsymbol{x}_p+r\underline{\eta} \right) +f\left( \boldsymbol{x}_p-r\underline{\eta} \right) \right) +\frac{1}{2}\underline{\omega}\underline{\eta}\left( f\left( \boldsymbol{x}_p-r\underline{\eta} \right) -f\left( \boldsymbol{x}_p+r\underline{\eta} \right) \right) .
	\end{align*}
\end{theorem}
Recall that the Cauchy kernel to monogenic functions over $\mathbb{R}^{p+2}$ is given by 
	\begin{align*}
		E\left( \boldsymbol{x} \right) =\frac{1}{\sigma _{p+1}}\frac{\overline{\boldsymbol{x}}}{\left| \boldsymbol{x} \right|^{p+2}},\ x\in \mathbb{R}^{p+2} \setminus \{0\} ,
	\end{align*}
	where $\sigma _{p+1}=2\frac{\Gamma ^{p+2}\left( \frac{1}{2} \right)}{\Gamma \left( \frac{p+2}{2} \right)}$ is the surface area of the unit sphere in $\mathbb{R}^{p+2}$. With the representation formula for generalized partial-slice monogenic functions given in Theorem \ref{*Representation Formula*}, it is natural to define the generalized partial-slice Cauchy kernel as follows.
\begin{definition}
	Given $\boldsymbol{y}\in \mathbb{R} ^{p+q+1}$, we call the function $\mathcal{E} _{\boldsymbol{y}}\left( \cdot \right) $ left generalized partial-slice Cauchy kernel defined by 
	\begin{eqnarray}
		\mathcal{E} _{\boldsymbol{y}}\left( \boldsymbol{x} \right) =\frac{1}{2}\left( 1-\underline{\omega}\underline{\eta} \right) E_{\boldsymbol{y}}\left( \boldsymbol{x}_p+r\underline{\eta} \right) +\frac{1}{2}\left( 1+\underline{\omega}\underline{\eta} \right) E_{\boldsymbol{y}}\left( \boldsymbol{x}_p-r\underline{\eta} \right) ,
	\end{eqnarray}
where $\underline{\omega},\underline{\eta}$ are defined as in Theorem \ref{*Representation Formula*}.
\end{definition}
With the Cauchy-Pompeiu formula in the theory of monogenic functions and the representation formula given in Theorem \ref{*Representation Formula*}, a Cauchy-Pompeiu formula in the context of partial-slice monogenic functions has been discovered in \cite{Xu} for the case when $\boldsymbol{x}\in U$. There is also a Cauchy integral formula for the case when $\boldsymbol{x}\in\mathbb{R}^{p+q+1}\backslash \overline{U}$, which is stated as follows.
\begin{theorem}[Cauchy integral formula for the exterior domain] Let $\Gamma$ denote a Jordan surface, partitioning the space into an exterior domain denoted as ${U_{\underline{\eta}}}^-$ and an interior domain ${U_{\underline{\eta}}}^+$, for some fixed $\underline{\eta} \in \mathbb{S}$. The orientation of $\Gamma$ is chosen such that its normal vector points towards ${U_{\underline{\eta}}}^-$. Consider the function $f$ that is left generalized partial-slice monogenic in ${U_{\underline{\eta}}}^-$, continuously differentiable in ${U_{\underline{\eta}}}^- \cup \Gamma$, and possesses a limit value $f(\infty)$ at $x = \infty$. Then, the following equation holds
	\begin{eqnarray*}
		\int_{\Gamma}{\mathcal{E} _{\boldsymbol{y}}\left( \boldsymbol{x} \right) n\left( \boldsymbol{y} \right) f\left( \boldsymbol{y} \right)}dS_{\underline{\eta}} \left( \boldsymbol{y} \right) =
		\begin{split}
			\begin{cases}
				-f\left( \boldsymbol{x} \right) +f\left( \infty \right) &,\boldsymbol{x}\in {U_{\underline{\eta}}}^-,\\
				f\left( \infty \right) &,\boldsymbol{x}\in {U_{\underline{\eta}}}^+,
			\end{cases}   
		\end{split}    
	\end{eqnarray*} 
	where $n(\boldsymbol{y}) = \sum_{i=0}^p{n_i(\boldsymbol{y}) e_i + n_{p+1}(\boldsymbol{y}) \underline{\eta}}$ represents the unit exterior normal vector to $\partial U_{\underline{\eta}}$ at $\boldsymbol{y}$.
\end{theorem}
\begin{proof}
	We choose a sphere $\Gamma _{\rho}=\left\{ \boldsymbol{x}:\left| \boldsymbol{x} \right|=\rho \right\} $ with sufficiently large radius $\rho$ which contains $\Gamma$ and ${U_{\underline{\eta}}}^+$, and consider the domain ${U_{\underline{\eta }}}_{\rho}:={U_{\underline{\eta }}}^-\cap \left\{ \left| \boldsymbol{x} \right|<\rho \right\}$. Its boundary is ${\partial U_{\underline{\eta }}}_{\rho}:=\Gamma _{\rho}\cup \left( -\Gamma \right) $ taking into account the orientation of $\Gamma$. Then the Cauchy integral formula for $\boldsymbol{x}\in {U_{\underline{\eta }}}_{\rho}$ yields 
	\begin{align*}
		f\left( \boldsymbol{x} \right)  =&-\int_{\Gamma}{\mathcal{E} _{\boldsymbol{y}}\left( \boldsymbol{x} \right) n\left( \boldsymbol{y} \right) f\left( \boldsymbol{y} \right) dS_{\underline{\eta}} \left( \boldsymbol{y} \right)}+f\left( \infty \right) \int_{\left| \boldsymbol{y} \right|=\rho}{\mathcal{E} _{\boldsymbol{y}}\left( \boldsymbol{x} \right) n\left( \boldsymbol{y} \right) dS_{\underline{\eta}} \left( \boldsymbol{y} \right)}\\
		&+\int_{\left| \boldsymbol{y} \right|=\rho}{\mathcal{E} _{\boldsymbol{y}}\left( \boldsymbol{x} \right) n\left( \boldsymbol{y} \right) \left( f\left( \boldsymbol{y} \right) -f\left( \infty \right) \right) dS_{\underline{\eta}} \left( \boldsymbol{y} \right)}\\
		=&-\int_{\Gamma}{\mathcal{E} _{\boldsymbol{y}}\left( \boldsymbol{x} \right) n\left( \boldsymbol{y} \right) f\left( \boldsymbol{y} \right) dS_{\underline{\eta}} \left( \boldsymbol{y} \right)}+f\left( \infty \right) +R,\\
		R:=&\int_{\left| \boldsymbol{y} \right|=\rho}{\mathcal{E} _{\boldsymbol{y}}\left( \boldsymbol{x} \right) n\left( \boldsymbol{y} \right) \left( f\left( \boldsymbol{y} \right) -f\left( \infty \right) \right) dS_{\underline{\eta}} \left( \boldsymbol{y} \right)}.
	\end{align*}
	From our assumption the inequality $\left| f\left( \boldsymbol{y} \right) -f\left( \infty \right) \right|<\varepsilon $ follows for sufficiently large $\rho$. If we assume moreover $\rho >2\left| \boldsymbol{x} \right|$ we get 
	\begin{align*}
		\left| E_{\boldsymbol{y}}\left( \boldsymbol{x} \right) \right|=\frac{1}{\sigma _{p+1}}\frac{1}{\left| \boldsymbol{y}-\boldsymbol{x} \right|^{p+1}}\leqslant \frac{1}{\sigma _{p+1}}\frac{2^{p+1}}{\rho ^{p+1}},
	\end{align*}
	because of $\left| \boldsymbol{y}-\boldsymbol{x} \right|\geqslant \left| \boldsymbol{y} \right|-\left| \boldsymbol{x} \right|>\displaystyle\frac{\rho}{2}$. Then, we have
	\begin{align*}
		\left| \mathcal{E}_{\boldsymbol{y}}\left( \boldsymbol{x} \right) \right|&=\left| \frac{1}{2}\left( 1-\underline{\omega}\underline{\eta} \right) E_{\boldsymbol{y}}\left( \boldsymbol{x}_p+r\underline{\eta} \right) +\frac{1}{2}\left( 1+\underline{\omega}\underline{\eta} \right) E_{\boldsymbol{y}}\left( \boldsymbol{x}_p-r\underline{\eta} \right) \right|\\
		&\leqslant \frac{1}{2}\left| \left( 1-\underline{\omega }\underline{\eta } \right) E_{\boldsymbol{y}}\left( \boldsymbol{x}_p+r\underline{\eta } \right) \right|+\frac{1}{2}\left| \left( 1+\underline{\omega }\underline{\eta } \right) E_{\boldsymbol{y}}\left( \boldsymbol{x}_p-r\underline{\eta } \right) \right|\\
		&\leqslant \frac{1}{2}\left| 1-\underline{\omega }\underline{\eta } \right|\frac{1}{\sigma _{p+1}}\frac{2^{p+1}}{\rho ^{p+1}}+\frac{1}{2}\left| 1+\underline{\omega }\underline{\eta } \right|\frac{1}{\sigma _{p+1}}\frac{2^{p+1}}{\rho ^{p+1}}\leqslant \frac{2}{\sigma _{p+1}}\frac{2^{p+1}}{\rho ^{p+1}}.
	\end{align*}
	Hence, we can obtain that
	\begin{align*}
		\left| R \right|&\leqslant \int_{\left| \boldsymbol{y} \right|=\rho}{\left| \mathcal{E} _{\boldsymbol{y}}\left( \boldsymbol{x} \right) \right|\left| n\left( \boldsymbol{y} \right) \right|\left| f\left( \boldsymbol{y} \right) -f\left( \infty \right) \right|dS_{\underline{\eta}} \left( \boldsymbol{y} \right)}\\
		&\leqslant \frac{\varepsilon \cdot 2^{p+2}}{\sigma _{p+1}}\int_{\left| \boldsymbol{t} \right|=1}{ds_{\boldsymbol{t}}}=\varepsilon \cdot 2^{p+2}.
	\end{align*}
	For $\rho \rightarrow \infty$  the value $\varepsilon$ can be chosen arbitrarily small, and the claim for $\boldsymbol{x}\in{U_{\underline{\eta}}}^-$ is proved.
	\par
	Our proof is also valid in the inner domain ${U_{\underline{\eta}}}^+$ as then the right-hand side contains the value $0$ instead of $f\left( \boldsymbol{x} \right) $ in the formulas above.
\end{proof}
\par
The integral 
\begin{align*}
\int_{\Gamma}{\mathcal{E} _{\boldsymbol{y}}\left( \boldsymbol{x} \right) n\left( \boldsymbol{y} \right) f\left( \boldsymbol{y} \right)}dS_{\underline{\eta}} \left( \boldsymbol{y} \right)
\end{align*}
in the previous theorem is the Cauchy-type integral in the context of generalized partial-slice monogenic functions. The boundary behavior of this integral is described by
a Plemelj-Sokhotski formula given below.
\begin{theorem}[Plemelj-Sokhotski formula]\label{PSF} Let the function $f=\mathcal{I} \left( F \right) \in \mathcal{G} \mathcal{S} \left( \Omega _D \right) $ with its stem function $F$ being H\"older continuous. Suppose $U$ is a p-symmetric domain in $\mathbb{R}^{p+q+1}$ such that $U_{\underline{\eta}} \subset \Omega_{\underline{\eta}}$ for some $\underline{\eta} \in \mathbb{S}$. Then, at any regular point $\boldsymbol{x}\in \partial U_{\underline{\eta}}$, we have
	\begin{align*}
		&n.t.\-\lim_{\boldsymbol{t}\rightarrow \boldsymbol{x}} \int_{\partial U_{\underline{\eta}}}{\mathcal{E} _{\boldsymbol{y}}\left( \boldsymbol{t} \right) n\left( \boldsymbol{y} \right) f\left( \boldsymbol{y} \right) dS_{\underline{\eta}} \left( \boldsymbol{y} \right)}\\
		=&\frac{1}{2}\left[ \pm f\left( \boldsymbol{x} \right) +2\int_{\partial U_{\underline{\eta}}}{\mathcal{E} _{\boldsymbol{y}}\left( \boldsymbol{x} \right) n\left( \boldsymbol{y} \right) f\left( \boldsymbol{y} \right) dS_{\underline{\eta}} \left( \boldsymbol{y} \right)} \right], 
	\end{align*}
	where the limit is taken non-tangentially (denoted as $n.t.\-\lim$) with $\boldsymbol{t} \in {U_{\underline{\eta}}}^{\pm}$, $U_{\underline{\eta}}^{+} = U_{\underline{\eta}}$, and $U_{\underline{\eta}}^{-} = \Omega_{\underline{\eta}} \backslash \overline{U_{\underline{\eta}}}$. The unit exterior normal vector to $\partial U_{\underline{\eta}}$ at $\boldsymbol{y}$ is given by
	$$n\left( \boldsymbol{y} \right) =\sum_{i=0}^p{n_i\left( \boldsymbol{y} \right) e_i+n_{p+1}\left( \boldsymbol{y} \right) \underline{\eta}}.$$
\end{theorem}
\begin{proof}
	Let $\boldsymbol{t}=\boldsymbol{t}_p+r^{\prime}\underline{\omega}\in {U_{\underline{\eta}}}^{\pm}$, $\boldsymbol{x}=\boldsymbol{x}_p+r\underline{\omega}\in \partial U_{\underline{\eta}}$, $\boldsymbol{x}^{\prime}=\left( \boldsymbol{x}_p,r \right) \in D$ and $\underline{\eta}\in \mathbb{S} $. Let 
	$$f\left( \boldsymbol{x} \right) =F_1\left( \boldsymbol{x}^{\prime} \right) +\underline{\omega}F_2\left( \boldsymbol{x}^{\prime} \right), $$
	where the stem function $F=F_1+iF_2$ is H\"older continuous. Indeed, According to the Splitting Lemma in \cite{Xu}, there exist H\"older continuous functions $F_{A}^{j}:D\longrightarrow \mathbb{R} _{p+1}\left( j=1,2 \right) $ such that 
	\begin{align*}
		F_j=\sum_{A=\left\{ i_1,\cdots ,i_s \right\} \subset \left\{ p+2,\cdots ,p+q \right\}}{F_{A}^{j}e_A},\ j=1,2.
	\end{align*}
	where $e_A=e_{i_1}\cdots e_{i_s}$, $A=\left\{ i_1,\cdots ,i_s \right\} \subset \left\{ p+2,\cdots ,p+q \right\}$ with $i_1<\cdots <i_s$, and $I_{\emptyset}=1$ when $A=\emptyset$.
	\par
	The Plemelj-Sokhotski formula for H\"older continuous functions in \cite{Gu} gives that 
	\begin{align*}
		&n.t.\-\lim_{\boldsymbol{t}\rightarrow \boldsymbol{x}} \int_{\partial U_{\underline{\eta }}}{E_{\boldsymbol{y}}\left( \boldsymbol{t}_p+r^{\prime}\underline{\eta } \right)}n\left( \boldsymbol{y} \right) F_{A}^{j}\left( \boldsymbol{y} \right) dS_{\underline{\eta}} \left( \boldsymbol{y} \right) \\&=\frac{1}{2}\left[ \pm F_{A}^{j}\left( \boldsymbol{x}_p+r\underline{\eta } \right) +2\int_{\partial U_{\underline{\eta }}}{E_{\boldsymbol{y}}\left( \boldsymbol{x}_p+r\underline{\eta } \right)}n\left( \boldsymbol{y} \right) F_{A}^{j}\left( \boldsymbol{y} \right) dS_{\underline{\eta}} \left( \boldsymbol{y} \right) \right],
	\end{align*}
	where the variable $\boldsymbol{y}=\boldsymbol{y}_p+\tilde{r}\underline{\eta}\in \partial U_{\underline{\eta}}$ should be interpreted as $\left( \boldsymbol{y}_p,\tilde{r} \right) \in \mathbb{R} ^{p+2}$. Hence, we have
	\begin{align*}
		&n.t.\-\lim_{\boldsymbol{t}\rightarrow \boldsymbol{x}}  \int_{\partial U_{\underline{\eta }}}{E_{\boldsymbol{y}}\left( \boldsymbol{t}_p+r^{\prime}\underline{\eta } \right)}n\left( \boldsymbol{y} \right) F\left( \boldsymbol{y} \right) dS_{\underline{\eta}} \left( \boldsymbol{y} \right) \\&=\frac{1}{2}\left[ \pm F\left( \boldsymbol{x}_p+r\underline{\eta } \right) +2\int_{\partial U_{\underline{\eta }}}{E_{\boldsymbol{y}}\left( \boldsymbol{x}_p+r\underline{\eta } \right)}n\left( \boldsymbol{y} \right) F\left( \boldsymbol{y} \right) dS_{\underline{\eta}} \left( \boldsymbol{y} \right) \right] .
	\end{align*}
	Substituting $\underline{\eta}$ for $i$ in the stem function $F$ within the aforementioned formula, we obtain
	\begin{align*}
		&n.t.\-\lim_{\boldsymbol{t}\rightarrow \boldsymbol{x}}  \int_{\partial U_{\underline{\eta }}}{E_{\boldsymbol{y}}\left( \boldsymbol{t}_p+r^{\prime}\underline{\eta } \right)}n\left( \boldsymbol{y} \right) f\left( \boldsymbol{y} \right) dS_{\underline{\eta}} \left( \boldsymbol{y} \right) \\&=\frac{1}{2}\left[ \pm f\left( \boldsymbol{x}_p+r\underline{\eta } \right) +2\int_{\partial U_{\underline{\eta }}}{E_{\boldsymbol{y}}\left( \boldsymbol{x}_p+r\underline{\eta } \right)}n\left( \boldsymbol{y} \right) f\left( \boldsymbol{y} \right) dS_{\underline{\eta}} \left( \boldsymbol{y} \right) \right] ,
	\end{align*}
	in other words,
	\begin{align*}
		f\left( \boldsymbol{x}_p+r\underline{\eta } \right) =&\pm 2\bigg[ n.t.\-\lim_{\boldsymbol{t}\rightarrow \boldsymbol{x}}  \int_{\partial U_{\underline{\eta }}}{E_{\boldsymbol{y}}\left( \boldsymbol{t}_p+r^{\prime}\underline{\eta } \right)}n\left( \boldsymbol{y} \right) f\left( \boldsymbol{y} \right) dS_{\underline{\eta}} \left( \boldsymbol{y} \right)\\
		&-\int_{\partial U_{\underline{\eta }}}{E_{\boldsymbol{y}}\left( \boldsymbol{x}_p+r\underline{\eta } \right)}n\left( \boldsymbol{y} \right) f\left( \boldsymbol{y} \right) dS_{\underline{\eta}} \left( \boldsymbol{y} \right) \bigg].
	\end{align*}
	Moreover, the formula remains valid when the variable $\boldsymbol{x}_p + r\underline{\eta}$ is substituted with $\boldsymbol{x}_p - r\underline{\eta}$. Applying the Representation formula stated in Theorem \ref{*Representation Formula*}, we then deduce that 
	\begin{align*}
		f\left( \boldsymbol{x}_p+r\underline{\omega} \right) =\frac{1}{2}\left( 1-\underline{\omega}\underline{\eta} \right) f\left( \boldsymbol{x}_p+r\underline{\eta} \right) +\frac{1}{2}\left( 1+\underline{\omega}\underline{\eta} \right) f\left( \boldsymbol{x}_p-r\underline{\eta} \right),
	\end{align*}
	from which we obtain the conclusion 
	\begin{align*}
		f\left( \boldsymbol{x} \right) =&\pm 2\bigg[ n.t.\-\lim_{\boldsymbol{t}\rightarrow \boldsymbol{x}}  \int_{\partial U_{\underline{\eta }}}{E_{\boldsymbol{y}}\left( \boldsymbol{t} \right)}n\left( \boldsymbol{y} \right) f\left( \boldsymbol{y} \right) dS_{\underline{\eta}} \left( \boldsymbol{y} \right) \\
		&-\int_{\partial U_{\underline{\eta }}}{E_{\boldsymbol{y}}\left( \boldsymbol{x} \right)}n\left( \boldsymbol{y} \right) f\left( \boldsymbol{y} \right) dS_{\underline{\eta}} \left( \boldsymbol{y} \right) \bigg] .
	\end{align*}
	Hence, we immediately have
	\begin{align*}
		&n.t.\-\lim_{\boldsymbol{t}\rightarrow \boldsymbol{x}} \int_{\partial U_{\underline{\eta}}}{\mathcal{E} _{\boldsymbol{y}}\left( \boldsymbol{t} \right) n\left( \boldsymbol{y} \right) f\left( \boldsymbol{y} \right) dS_{\underline{\eta}} \left( \boldsymbol{y} \right)}\\
		=&\frac{1}{2}\left[ \pm f\left( \boldsymbol{x} \right) +2\int_{\partial U_{\underline{\eta}}}{\mathcal{E} _{\boldsymbol{y}}\left( \boldsymbol{x} \right) n\left( \boldsymbol{y} \right) f\left( \boldsymbol{y} \right) dS_{\underline{\eta}} \left( \boldsymbol{y} \right)} \right],
	\end{align*}
	which completes the proof.
\end{proof}

\section{Norm estimates of the Teodorescu transform}
In \cite{Xu}, a global non-constant coefficients differential operator for $C^1$ function $f:\Omega \longrightarrow \mathbb{R}_{p+q}$ is given by
\begin{align*}
	\bar{\vartheta}f\left( \boldsymbol{x} \right) =D_{\boldsymbol{x}_p}f\left( \boldsymbol{x} \right) +\frac{\underline{\boldsymbol{x}}_q}{\left| \underline{\boldsymbol{x}}_q \right|^2}\mathbb{E} _{\underline{\boldsymbol{x}}_q}f\left( \boldsymbol{x} \right).
\end{align*}
The differential operator necessitates a more cautious approach due to the singularities it introduces for the $\left| \underline{\boldsymbol{x}}_q \right|^2$ term within the operator. Therefore, we will adopt the notation $\mathbb{R} _{*}^{p+q+1}:=\mathbb{R} ^{p+q+1}\backslash \mathbb{R} ^{p+1}$ for the remainder of this article.
\par
The Cauchy kernel for generalized partial-slice monogenic functions is defined as follows 
\begin{align*}
	K_{\boldsymbol{y}}\left( \boldsymbol{x} \right) =\frac{\mathcal{E} _{\boldsymbol{y}}\left( \boldsymbol{x} \right)}{\sigma _{q-1}\vert \underline{\boldsymbol{y}}_q \vert^{q-1}}
\end{align*}
where $\sigma_{q-1}$ is the area of the $\left( q-1 \right) \text{-}$sphere $\mathbb{S}$.
\par
The Cauchy-Pompeiu formula we mentioned earlier is only applicable to $\Omega _{\underline{\eta}}$. However, by the approach applied in \cite[Theorem 3.5, 3.6]{Di}, we can easily derive the Cauchy-Pompeiu formula and the Cauchy integral formula on the domain $\Omega _D$ as follows.
\begin{theorem}[Cauchy-Pompeiu formula]\label{CPFG}Let $\Omega _D\subset \mathbb{R} _{*}^{p+q+1}$ be a bounded domain as previously defined, and generalized partial-slice function $f\in ker\bar{\vartheta}$. If $U$ is a domain in $\mathbb{R}^{p+q+1}$ such that $U_D\subset \Omega _D$ is a bounded domain in $\mathbb{R}^{p+2}$ with smooth boundary $\partial U_D\subset \Omega _D$, then for any $\boldsymbol{x}\in U$, we have
	\begin{align*}
		f\left( \boldsymbol{x} \right) =\int_{\partial U_D}{K_{\boldsymbol{y}}\left( \boldsymbol{x} \right) n\left( \boldsymbol{y} \right) f\left( \boldsymbol{y} \right) dS\left( \boldsymbol{y} \right)}-\int_{U_D}{K_{\boldsymbol{y}}\left( \boldsymbol{x} \right) \left( \bar{\vartheta}f \right) \left( \boldsymbol{y} \right) d\sigma \left( \boldsymbol{y} \right)},
	\end{align*}
	where $n\left( \boldsymbol{y} \right)$ is the unit exterior normal vector to $\partial U_D$ at $\boldsymbol{y}$, $dS$ and $d\sigma$  stand for the classical Lebesgue surface element and volume element in $\mathbb{R}^{p+2}$, respectively.
\end{theorem}
\begin{theorem}[Cauchy integral formula]Let $\Omega _D\subset \mathbb{R} _{*}^{p+q+1}$ be a bounded domain as previously defined, and slice function $f\in ker\bar{\vartheta}$. Then, for any $\boldsymbol{x}\in \Omega _D$, we have 
\begin{align}\label{CIF}
	\int_{\partial \Omega _D}{K_{\boldsymbol{y}}\left( \boldsymbol{x} \right) n\left( \boldsymbol{y} \right) f\left( \boldsymbol{y} \right) d\sigma \left( \boldsymbol{y} \right)}=f\left( \boldsymbol{x} \right),
\end{align}
where $n\left( \boldsymbol{y} \right)$ represents the unit normal vector pointing outwards from the boundary $\partial \Omega _D$ at each point $\boldsymbol{y}$.
\end{theorem}
\par
If we denote 
\begin{align*}
	&T_{\Omega _D}f\left( \boldsymbol{x} \right) =-\int_{\Omega _D}{K_{\boldsymbol{y}}\left( \boldsymbol{x} \right) f\left( \boldsymbol{y} \right) d\sigma \left( \boldsymbol{y} \right)},\\
	&F_{\partial \Omega _D}f\left( \boldsymbol{x} \right) =\int_{\partial \Omega _D}{K_{\boldsymbol{y}}\left( \boldsymbol{x} \right) n\left( \boldsymbol{y} \right) f\left( \boldsymbol{y} \right) dS\left( \boldsymbol{y} \right)},
\end{align*}
then the equation in Cauchy-Pompeiu formula in Theorem \ref{CPFG} can be rewritten as 
\begin{align*}
	F_{\partial \Omega _D}f\left( \boldsymbol{x} \right) +T_{\Omega _D}\left( \bar{\vartheta}f \right) \left( \boldsymbol{x} \right) =f\left( \boldsymbol{x} \right) ,
\end{align*}for $\boldsymbol{x}\in \Omega _D$. Here, $T_{\Omega _D}$ is usually called the $\textit{Teodorescu transform}$. In particular, for functions with compact support in $\Omega_D$, we get $F_{\partial \Omega _D}f\left( \boldsymbol{x} \right) =0$. Hence, we have that 
\begin{align*}
	T_{\Omega _D}\left( \bar{\vartheta}f \right) \left( \boldsymbol{x} \right) =f\left( \boldsymbol{x} \right) ,
\end{align*}
which suggests that $T_{\Omega _D}$ is a left inverse of $\bar{\vartheta}$ when acts on functions with compact support. 
Now, we introduce the existence for $T_{\Omega_D}f$ and a norm estimate for $T_{\Omega_D}$ as follows.
\begin{proposition} \label{estimated}
	Suppose that $\Omega _D\subset \mathbb{R} _{*}^{p+q+1}$ is a bounded $p$-symmetric domain and $f\in L^t\left( \Omega _D \right) $. Then,
	\begin{enumerate}
		\item The integral $T_{\Omega _D}f\left( \boldsymbol{x} \right) $ exists everywhere in $\mathbb{R} _{*}^{p+q+1}$ when $t>q$;
		\item $\bar{\vartheta}T_{\Omega _D}f=0$ in $\mathbb{R} _{*}^{p+q+1}\backslash \overline{\Omega _D}$;
		\item Further, when $t>\max\{2p-1,2q-1\}$ and $q>1$, we have 
		\begin{align*}
			\left\| T_{\Omega _D}f \right\| _{L^t}\leqslant C\left( t,q,\Omega _D \right) \left\| f \right\| _{L^t}.
		\end{align*}
	\end{enumerate}
\end{proposition}
\begin{proof}
	In the proof below, we consistently use the letter $C$ to represent various finite constants. Firstly, we observe that 
	\begin{align*}
	    \left| T_{\Omega _D}f\left( \boldsymbol{x} \right) \right|&=\left| -\int_{\Omega _D}{K_{\boldsymbol{y}}\left( \boldsymbol{x} \right) f\left( \boldsymbol{y} \right) d\sigma \left( \boldsymbol{y} \right)} \right|\\
		&\leqslant \int_{\Omega _D}{\left| K_{\boldsymbol{y}}\left( \boldsymbol{x} \right) f\left( \boldsymbol{y} \right) \right|d\sigma \left( \boldsymbol{y} \right)}\\
		&\leqslant \left( \int_{\Omega _D}{\left| K_{\boldsymbol{y}}\left( \boldsymbol{x} \right) \right|^s d\sigma \left( \boldsymbol{y} \right)} \right) ^{\frac{1}{s}}\left\| f \right\| _{L^t},  
	\end{align*}
	where $\frac{1}{s}+\frac{1}{t}=1$, $s,t>1$. Now, we consider 
	\begin{align*}
		&\int_{\Omega _D}{\left| K_{\boldsymbol{y}}\left( \boldsymbol{x} \right) \right|^s d\sigma \left( \boldsymbol{y} \right)}\\
		=&C\int_{\mathbb{S} ^+}{\int_{\Omega _{\underline{\eta }}}{\left| \frac{\mathcal{E} _{\boldsymbol{y}}\left( \boldsymbol{x} \right)}{\vert \underline{\boldsymbol{y}}_q \vert^{q-1}} \right|}^s}\cdot\vert \underline{\boldsymbol{y}}_q \vert^{q-1} d\sigma _{\underline{\eta }}\left( \boldsymbol{y} \right) dS({\underline{\eta }}) \\
		=&C\int_{\mathbb{S} ^+}{\int_{\Omega _{\underline{\eta }}}{\left| \alpha E_{\boldsymbol{y}}\left( \boldsymbol{x}_p+\underline{\eta }r \right) +\beta E_{\boldsymbol{y}}\left( \boldsymbol{x}_p-\underline{\eta }r \right) \right|}^s}\\
		&\quad\quad\quad\quad\cdot \vert \underline{\boldsymbol{y}}_q \vert^{-\left( s-1 \right) \left( q-1 \right)}d\sigma _{\underline{\eta}}\left( \boldsymbol{y} \right) dS({\underline{\eta }})  \\
		\leqslant& C\int_{\mathbb{S} ^+}{\int_{\Omega _{\underline{\eta }}}{\left( \left| \boldsymbol{y}-\left( \boldsymbol{x}_p+\underline{\eta }r \right) \right|^{-s\left( p+1 \right)}+\left| \boldsymbol{y}-\left( \boldsymbol{x}_p-\underline{\eta }r \right) \right|^{-s\left( p+1 \right)} \right)}}\\
		&\quad\quad\quad\quad\cdot \vert \underline{\boldsymbol{y}}_q \vert^{-\left( s-1 \right) \left( q-1 \right)}d\sigma _{\underline{\eta}}\left( \boldsymbol{y} \right)dS({\underline{\eta }})  ,
	\end{align*}
	where $\alpha =\frac{1-\underline{\omega}\underline{\eta}}{2}$, $\beta =\frac{1+\underline{\omega}\underline{\eta}}{2}$, $\mathbb{S} ^+$ is the half unit sphere of $\mathbb{S}$.
	\par
	Next, we only need to verify that the integral 
	\begin{align}\label{Int}
		\int_{\mathbb{S} ^+}{\int_{\Omega _{\underline{\eta }}}{\left| \boldsymbol{y}-\left( \boldsymbol{x}_p+\underline{\eta }r \right) \right|^{-s\left( p+1 \right)}}}\cdot \vert \underline{\boldsymbol{y}}_q \vert^{-\left( s-1 \right) \left( q-1 \right)}d\sigma _{\underline{\eta }}\left( \boldsymbol{y} \right)dS({\underline{\eta }}) 
	\end{align}
	is finite, the argument for the other one is similar.
	\par
	We notice that $\Omega _D\subset \mathbb{R} _{*}^{p+q+1}$ is bounded, let $\boldsymbol{u}_p=(u_0,\cdots,u_p)$, and
	\begin{align*}
		E=\left\{ \boldsymbol{x}=\boldsymbol{u}_p+\underline{\eta }v:r_0<u_0,\cdots,u_p<R,0<v<M,\underline{\eta }\in \mathbb{S} \right\},
	\end{align*}
	then, there exist $r,R,M>0$ such that $\Omega _D\subset E$. Hence, we have 
	\begin{align*}
		\int_{\mathbb{S} ^+}{\int_{\Omega _{\underline{\eta }}}{\left| \boldsymbol{y}-\left( \boldsymbol{x}_p+\underline{\eta }r \right) \right|^{-s\left( p+1 \right)}}}\cdot \vert \underline{\boldsymbol{y}}_q \vert^{-\left( s-1 \right) \left( q-1 \right)}d\sigma _{\underline{\eta }}\left( \boldsymbol{y} \right) dS({\underline{\eta }})  \\
		\leqslant \int_{\mathbb{S} ^+}{\int_{E_{\underline{\eta }}}{\left| \boldsymbol{y}-\left( \boldsymbol{x}_p+\underline{\eta }r \right) \right|^{-s\left( p+1 \right)}}}\cdot \vert \underline{\boldsymbol{y}}_q \vert^{-\left( s-1 \right) \left( q-1 \right)}d\sigma _{\underline{\eta }}\left( \boldsymbol{y} \right) dS({\underline{\eta }}) .
	\end{align*}
	We can easily find that all $E_{\underline{\eta}}$ are the same with $\underline{\eta}\in \mathbb{S} $ up to a rotation. Hence, we only need to show that 
	\begin{eqnarray} \label{proposition 2}
		\int_{E_{\underline{\eta }}}{\left| \boldsymbol{y}-\left( \boldsymbol{x}_p+\underline{\eta }r \right) \right|^{-s\left( p+1 \right)}\cdot \vert \underline{\boldsymbol{y}}_q \vert^{-\left( s-1 \right) \left( q-1 \right)}d\sigma _{\underline{\eta }}\left( \boldsymbol{y} \right)}
	\end{eqnarray}
	is finite uniformly with respect to $\underline{\eta}$. The singularities of the integral above occur in the following two cases.
	\begin{enumerate}
		\item The singular point $\boldsymbol{y}=\boldsymbol{x}_p+\underline{\eta}r$,
		\item Points $\boldsymbol{x}$ with $\underline{\boldsymbol{x}}_q=0$.
	\end{enumerate}
	Let $B\left( \boldsymbol{x}_p+\underline{\eta}r,\epsilon _0 \right) \subset E_{\underline{\eta}}$ be a neighborhood of $\boldsymbol{x}_p+\underline{\eta}r$ with a sufficiently small $\epsilon _0$ and 
	\begin{align*}
		 E_{\epsilon _1}=\left\{ \boldsymbol{u}_p+\underline{\eta }v:r_0<u_0,\cdots,u_p<R,-\epsilon _1<v<\epsilon _1 \right\} .
	\end{align*}
	The finiteness of the integral $\left( \ref{proposition 2} \right)$ is equivalent the finiteness of 
	\begin{align*}
		\int_{B\left( \boldsymbol{x}_p+\underline{\eta }r,\epsilon _0 \right) \cup E_{\epsilon _1}}{\left| \boldsymbol{y}-\left( \boldsymbol{x}_p+\underline{\eta }r \right) \right|^{-s\left( p+1 \right)}\cdot\vert \underline{\boldsymbol{y}}_q \vert^{-\left( s-1 \right) \left( q-1 \right)}d\sigma _{\underline{\eta }}\left( \boldsymbol{y} \right)}.
	\end{align*}
	On the one hand, let $\boldsymbol{y}=s\zeta \, ,\, \zeta \in \mathbb{S}$, then we have 
	\begin{align*}
		&\int_{B\left( \boldsymbol{x}_p+\underline{\eta }r,\epsilon _0 \right)}{\left| \boldsymbol{y}-\left( \boldsymbol{x}_p+\underline{\eta }r \right) \right|^{-s\left( p+1 \right)}\cdot \vert \underline{\boldsymbol{y}}_q \vert^{-\left( s-1 \right) \left( q-1 \right)}d\sigma _{\underline{\eta }}\left( \boldsymbol{y} \right)}\\
		& \leqslant C\left( \boldsymbol{x},t,q,\Omega _D \right) \int_0^{\epsilon _0}{\int_{\mathbb{S}}{r^{p+1-s\left( p+1 \right)}dsdS\left( \zeta \right)}<\infty},
	\end{align*}
	with $t>q$. On the other hand, recall that  $\boy=\boldsymbol{u}_p+\underline{\eta }v\in E_{\epsilon _1}$, then we obtain 
	\begin{align*}
		&\int_{E_{\epsilon _1}}{\left| \boldsymbol{y}-\left( \boldsymbol{x}_p+\underline{\eta }r \right) \right|^{-s\left( p+1 \right)}\cdot \vert \underline{\boldsymbol{y}}_q \vert^{-\left( s-1 \right) \left( q-1 \right)}d\sigma _{\underline{\eta }}\left( \boldsymbol{y} \right)}\\
		&\leqslant C\left( \boldsymbol{x},t,q,\Omega _D \right) \int_{r_0}^R\cdots\int_{r_0}^R{\int_0^{\epsilon _1}{v^{-\left( q-1 \right) \left( s-1 \right)}du_0\cdots du_pdv}<\infty}.
	\end{align*}
	Therefore, we obtain $Tf(\bx)$ is finite for all $\bx\in\Omega_D$, which can lead to the fact that $T_{\Omega _D}f\left( \boldsymbol{x} \right) $ exists everywhere in $\mathbb{R} _{*}^{p+q+1}$. Further, as $T_{\Omega _D}f$ has no singular points in $\mathbb{R} _{*}^{p+q+1}\backslash \overline{\Omega _D}$ and $\bar{\vartheta}\mathcal{E} _{\boldsymbol{y}}\left( \boldsymbol{x} \right) =0$ is trivial, it follows that $\bar{\vartheta}T_{\Omega _D}f=0$ in $\mathbb{R} _{*}^{p+q+1}\backslash \overline{\Omega _D}$.
	\par
	It is worth pointing out that we can not obtain the third statement by the estimate for $|T_{\Omega_D}f(\boldsymbol{x})|$ obtained above, since the constant $C$ depends on $\boldsymbol{x}$. Here, we need a more subtle argument to deal with the singularities in \eqref{Int}. 
	\par
	We denote $\bx_{\underline{\eta }}=\boldsymbol{x}_p+r\underline{\eta }$, and let $r_1>0$ be sufficiently large such that $\Omega_D\subset B(\bx,r_1)$ for all $\bx\in\Omega_D$. Let $\boy=\bx_{\ueta}+l\underline{\gamma}=\boldsymbol{x}_p+r\underline{\eta }+l\underline{\gamma}$ with $\underline{\gamma}\in\mathbb{S}$, $B_{\ueta}(\bx_{\ueta},r_1):=B(\bx_{\ueta},r_1)\cap\mathbb{C}_{\ueta}$. Then, we can see that $\vert \underline{\boldsymbol{y}}_q \vert=|r+l\cos\theta|=|r+l\langle\ueta,\ugamma\rangle|$, where $\theta=\arccos\langle\ueta,\ugamma\rangle$ is the angle between $\ueta$ and $\ugamma$. Now, we have
	\begin{align*}
	&\int_{\mathbb{S} ^+}{\int_{\Omega _{\underline{\eta }}}{\vert \boldsymbol{y}-\bx_{\ueta} |^{-s\left( p+1 \right)}}}\cdot \vert \underline{\boldsymbol{y}}_q \vert^{-\left( s-1 \right) \left( q-1 \right)}d\sigma _{\underline{\eta }}\left( \boldsymbol{y} \right) dS({\underline{\eta }}) \\
	\leq&\int_{\mathbb{S} ^+}{\int_{B_{\ueta}(\bx_{\ueta},r_1)}{| \boldsymbol{y}-\bx_{\ueta} |^{-s\left( p+1 \right)}}}\cdot \vert \underline{\boldsymbol{y}}_q \vert^{-\left( s-1 \right) \left( q-1 \right)}d\sigma _{\underline{\eta }}\left( \boldsymbol{y} \right) dS({\underline{\eta }}) \\
	\leq&\int_{\mathbb{S} ^+}\int_{\mathbb{S} ^+}\int_0^{r_1}l^{-s(p+1)}|r+l\langle\ueta,\ugamma\rangle|^{-\left( s-1 \right) \left( q-1 \right)}l^{p+1}dldS(\uomega)dS(\ueta)\\
	\leq&C\int_0^{r_1}l^{(1-s)(p+1)}|r-l|^{(1-s)(q-1)}dl.
	\end{align*}
To estimate the $L^t$ norm of $T_{\Omega_D}$, we only need to show that the $L^t$ norm of \eqref{Int} is finite. Recall that $r=|\underline{\bx}_{q}|$, and with the argument above, we  calculate
\begin{align*}
&\int_{\Omega_D}\bigg\vert\int_0^{r_1}l^{(1-s)(p+1)}|r-l|^{(1-s)(q-1)}dl\bigg\vert^tdV(\bx)\\
\leq&\int_{\So^+}\int_{\Omega_{\uomega}}\bigg\vert\int_0^{r_1}l^{(1-s)(p+1)}|r-l|^{(1-s)(q-1)}dl\bigg\vert^tr^{q-1}dV_{\uomega}(\bx)dS(\uomega).
\end{align*}
Let $a,b>0$ such that $$\Omega_{\uomega}\subset\{\bx=\sum_{i=0}^pe_ix_i+r\uomega\in\mathbb{R}^{p+q+1}: -a<x_1,\ldots,x_p<a,0<r<b\}$$ for all $\uomega\in\mathbb{S}$. Therefore, the last integral above becomes
\begin{align}\label{Eva}
\leq&\int_{\So^+}\int_{-a}^a\cdots\int_{-a}^a\int_{0}^b\bigg\vert\int_0^{r_1}l^{(1-s)(p+1)}|r-l|^{(1-s)(q-1)}dl\bigg\vert^tr^{q-1}\nonumber\\
&\cdot dx_0\cdots dx_pdrdS(\uomega)\nonumber\\
\leq&C\int_{0}^b\bigg\vert\int_0^{r_1}l^{(1-s)(p+1)}|r-l|^{(1-s)(q-1)}dl\bigg\vert^tr^{q-1}dr,
\end{align}
where the last equation comes from the fact that $r=|\ubx_q|$ and $l$ are independent to $x_0,\ldots,x_p$. Now, let $r-l=\sigma$, we calculate
\begin{align*}
&\int_0^{r_1}l^{(1-s)(p+1)}|r-l|^{(1-s)(q-1)}dl\\
\leq&\frac{1}{2}\int_0^{r_1}l^{2(1-s)(p+1)}+|r-l|^{2(1-s)(q-1)}dl\\
\leq&C\bigg[\int_0^{r_1}l^{2(1-s)(p+1)}dl+\int^r_{r-r_1}|\sigma|^{2(1-s)(q-1)}dl\bigg]\\
\leq&C\bigg[\int_0^{r_1}l^{2(1-s)(p+1)}dl+\int^r_0\sigma^{2(1-s)(q-1)}dl-\int^0_{r-r_1}\sigma^{2(1-s)(q-1)}dl\\
\leq &C(r_1,t,p)+\frac{1}{2(1-s)(q-1)+1}\big(r^{2(1-s)(q-1)+1}+(r-r_1)^{2(1-s)(q-1)+1}\big),
\end{align*}
where the last inequality requires $2(1-s)(p+1)>-1$ and $2(1-s)(q-1)>-1$, which is $t>\max\{2p-1,2q-1\}$. Plugging back to \eqref{Eva} to obtain
\begin{align*}
&\int_{\Omega_D}\bigg\vert\int_0^{r_1}l^{(1-s)(p+1)}|r-l|^{(1-s)(q-1)}dl\bigg\vert^tdV(\bx)\\
\leq&C\int_0^b\bigg\vert C(r_1,t,p)+\frac{r^{2(1-s)(q-1)+1}+(r-r_1)^{2(1-s)(q-1)+1}}{2(1-s)(q-1)+1}\bigg\vert^t r^{q-1}dr\\
\leq&C(r_1,t,p)+C_1(r_1,t,p)\bigg[\int_0^br^{2(1-s)(q-1)t+t+q-1}dr\\
&+\int_0^b(r-r_1)^{2(1-s)(q-1)t+t}r^{q-1}dr\bigg],
\end{align*}
where the last integral is finite when $2(1-s)(q-1)t+t>-1$, which also implies that $2(1-s)(q-1)t+t+q-1>-1$. A straightforward calculation shows that this can be satisfied when $t>\max\{2p-1,2q-1\}$ and $q>1$, which completes the proof.
\end{proof}
Further, we claim that $T_{\Omega}$ also maps slice functions to slice functions as follows.
\begin{proposition}
	Let $\Omega _D\subset \mathbb{R} _{*}^{p+q+1}$ be a bounded domain and $t>q$, then we have that
	\begin{align*}
		T_{\Omega _D}\,:\,\mathcal{L} ^t\left( \Omega _D \right) \longrightarrow \mathcal{L} ^t\left( \Omega _D \right) 
	\end{align*}
	is continuous. 
\end{proposition}
\begin{proof}
	It is easy to know that $T_{\Omega _D}$ maps $L^t\left( \Omega _D \right)$ to $L^t\left( \Omega _D \right)$ can be obtained immediately from Proposition \ref{estimated}. Here, we only prove $T_{\Omega _D}$ also maps $L^t\left( \Omega _D \right)$ to $\mathcal{GS} \left( \Omega _D \right)$. Recall that $\left[ \boldsymbol{y} \right] =\left\{ \boldsymbol{y}_p+r\underline{\omega},\underline{\omega}\in \mathbb{S} \right\}$ and observe that $\mathcal{E} _{\boldsymbol{y}}\left( \boldsymbol{x} \right)$ is left generalized partial-slice monogenic in $\mathbb{R} ^{p+q+1}\backslash \left[ \boldsymbol{y} \right]$. We assume that $\mathcal{E} _{\boldsymbol{y}}\left( \boldsymbol{x} \right) =\mathcal{I} \left( F \right)$, $\boldsymbol{x}^{\prime}=\left( \boldsymbol{x}_p,r \right) \in D$. Set 
	\begin{align*}
		\mathcal{E} _{\boldsymbol{y}}\left( \boldsymbol{x} \right) =F_1\left( \boldsymbol{x}^{\prime} \right) +\underline{\omega }F_2\left( \boldsymbol{x}^{\prime} \right),
	\end{align*}
	where $F=F_1+iF_2$ is the stem function which induces the generalized partial-slice monogenic function $\mathcal{E} _{\boldsymbol{y}}\left( \boldsymbol{x} \right) $ as in the definition. Therefore, we have 
	\begin{align*}
		T_{\Omega _D}f\left( \boldsymbol{x} \right) 
		&=-\int_{\Omega _D}{K _{\boldsymbol{y}}\left( \boldsymbol{x} \right) f\left( \boldsymbol{y} \right) d\sigma \left( \boldsymbol{y} \right)}\\
		&=-\int_{\Omega _D}{\frac{\mathcal{E} _{\boldsymbol{y}}\left( \boldsymbol{x} \right)}{\sigma _{q-1}\vert \underline{\boldsymbol{y}}_q \vert^{q-1}}f\left( \boldsymbol{y} \right) d\sigma \left( \boldsymbol{y} \right)}\\
		&=-\frac{1}{\sigma _{q-1}}\left[ \int_{\Omega _D}{\frac{F_1\left( \boldsymbol{x}^{\prime} \right)}{\vert \underline{\boldsymbol{y}}_q \vert^{q-1}} f\left( \boldsymbol{y} \right) d\sigma \left( \boldsymbol{y} \right)}+\underline{\omega }\int_{\Omega _D}{\frac{F_2\left( \boldsymbol{x}^{\prime} \right)}{\vert \underline{\boldsymbol{y}}_q \vert^{q-1}} f\left( \boldsymbol{y} \right) d\sigma \left( \boldsymbol{y} \right)} \right] .
	\end{align*}
	If we let 
	\begin{align*}
		&H\left( \boldsymbol{x}^{\prime} \right) =H_1\left( \boldsymbol{x}^{\prime} \right) +iH_2\left( \boldsymbol{x}^{\prime} \right)\\ 
		=&\left[-\frac{1}{\sigma _{q-1}} \int_{\Omega _D}{\frac{F_1\left( \boldsymbol{x}^{\prime} \right)}{\vert \underline{\boldsymbol{y}}_q \vert^{q-1}} f\left( \boldsymbol{y} \right) d\sigma \left( \boldsymbol{y} \right)} \right] +i\left[ -\frac{1}{\sigma _{q-1}}\int_{\Omega _D}{\frac{F_2\left( \boldsymbol{x}^{\prime} \right)}{\vert \underline{\boldsymbol{y}}_q \vert^{q-1}} f\left( \boldsymbol{y} \right) d\sigma \left( \boldsymbol{y} \right)} \right] ,   
	\end{align*}
	then since $F_1\left( \boldsymbol{x}^{\prime} \right) $, $F_2\left( \boldsymbol{x}^{\prime} \right)$, $f\left( \boldsymbol{x} \right)$ are all $\mathbb{R} _{p+q}$-valued, we know that $H_1\left( \boldsymbol{x}^{\prime} \right) $, $H_2\left( \boldsymbol{x}^{\prime} \right) $ are both real $\mathbb{R} _{p+q}$-valued. Further, since $F$ is the stem function, we have 
	\begin{align*}
		F_1\left( \boldsymbol{x}_p,-r \right) =F_1\left( \boldsymbol{x}_p,r \right) \,\,and\,\,F_2\left( \boldsymbol{x}_p,-r \right) =-F_2\left( \boldsymbol{x}_p,r \right).
	\end{align*}
	So we get 
	\begin{align*}
		H_1\left( \boldsymbol{x}_p,-r \right) =H_1\left( \boldsymbol{x}_p,r \right) \,\,and\,\,H_2\left( \boldsymbol{x}_p,-r \right) =-H_2\left( \boldsymbol{x}_p,r \right),
	\end{align*}
	and $H\left( \boldsymbol{x}^{\prime} \right)$ is the stem function. Hence, the function $T_{\Omega _D}f$ induced by $H\left( \boldsymbol{x}^{\prime} \right) $ is a generalized partial-slice function, which completes the proof.
\end{proof}

\subsection*{Conclusion}
In this article, we present some integral formulas, such as the Cauchy integral formula for the exterior domain and the Plemelj-Sokhotski formula in the theory of generalized partial-slice monogenic functions. Further, we start an investigation to the Teodorescu transform in this context and a norm estimation for the Teodorescu transform is introduced as well. This leads to a further study to some questions related to the Teodorescu transform. More specifically, for instance, a Hodge decomposition, a generalized $\Pi$ operator and a Vekua system can be investigated in the framework of generalized partial-slice monogenic functions. 

\subsection*{Acknowledgments}
This article is dedicated to the memory of Professor Yuri Grigoriev. The work of Chao Ding is supported by National Natural Science Foundation of China (No. 12271001), Natural Science Foundation of Anhui Province (No. 2308085MA03) and Excellent University Research and Innovation Team in Anhui Province (No. 2024AH010002).


\subsection*{Data Availability}
No new data were created or analysed during this study. Data sharing is not applicable to this article.



\end{document}